\documentclass[12pt]{amsart}
\usepackage[margin=1.35in]{geometry}
\usepackage{amscd,amsmath,amsxtra,amsthm,amssymb,stmaryrd,xr,mathrsfs,mathtools,enumerate,commath, comment}
\usepackage{stmaryrd}
\usepackage{multirow}
\usepackage{xcolor}
\usepackage{commath}
\usepackage{comment}
\usepackage{tikz-cd}
\usepackage{longtable} 
\usepackage{pdflscape} 
\usepackage{booktabs}
\usepackage{hyperref}
\definecolor{vegasgold}{rgb}{0.77, 0.7, 0.35}
\definecolor{darkgoldenrod}{rgb}{0.72, 0.53, 0.04}
\definecolor{gold(metallic)}{rgb}{0.83, 0.69, 0.22}
\hypersetup{
 colorlinks=true,
 linkcolor=darkgoldenrod,
 filecolor=brown,      
 urlcolor=gold(metallic),
 citecolor=darkgoldenrod,
 }
\newtheorem{lthm}{Theorem}

\usepackage[all,cmtip]{xy}

\DeclareFontFamily{U}{wncy}{}
\DeclareFontShape{U}{wncy}{m}{n}{<->wncyr10}{}
\DeclareSymbolFont{mcy}{U}{wncy}{m}{n}
\DeclareMathSymbol{\Sh}{\mathord}{mcy}{"58}
\usepackage[T2A,T1]{fontenc}
\usepackage[OT2,T1]{fontenc}

\newtheorem{theorem}{Theorem}[section]
\newtheorem{lemma}[theorem]{Lemma}

\newtheorem*{theorem*}{Theorem}
\newtheorem*{ass*}{Assumption}
\newtheorem{definition}[theorem]{Definition}
\newtheorem{corollary}[theorem]{Corollary}

\newtheorem{proposition}[theorem]{Proposition}
\newtheorem{question}[theorem]{Question}

\newcommand{\cF}{\mathcal{F}}

\newcommand{\cK}{\mathcal{K}}

\newcommand{\cG}{\mathcal{G}}

\newcommand{\Z}{\mathbb{Z}}
\newcommand{\p}{\mathfrak{p}}
\newcommand{\Q}{\mathbb{Q}}
\newcommand{\F}{\mathbb{F}}

\newcommand{\cL}{\mathcal{L}}

\newcommand{\cO}{\mathcal{O}}

\newcommand{\op}[1]{\operatorname{#1}}

\newcommand\mtx[4] { \left( {\begin{array}{cc}
 #1 & #2 \\
 #3 & #4 \\
 \end{array} } \right)}

\numberwithin{equation}{section}

\begin{document}

\title[Rank distribution in cubic twist families]{Rank distribution in cubic twist families of elliptic curves}

\author[A.~Ray]{Anwesh Ray}
\address[Ray]{Chennai Mathematical Institute, H1, SIPCOT IT Park, Kelambakkam, Siruseri, Tamil Nadu 603103, India}
\email{anwesh@cmi.ac.in}

\author[P.~Shingavekar]{Pratiksha Shingavekar}
\address[Shingavekar]{Chennai Mathematical Institute, H1, SIPCOT IT Park, Kelambakkam, Siruseri, Tamil Nadu 603103, India}
\email{pshingavekar@gmail.com}

\keywords{Elliptic curves, cubic twist families, rank distribution, Selmer groups and stability, explicit computation}
\subjclass[2020]{11G05, 11R45, 11R34}

\maketitle

\begin{abstract}
Let $a$ be an integer which is not of the form $n^2$ or $-3 n^2$ for $n\in \mathbb{Z}$. Let $E_a$ be the elliptic curve with rational $3$-isogeny defined by $E_a:y^2=x^3+a$, and $K:=\mathbb{Q}(\mu_3)$. Assume that the $3$-Selmer group of $E_a$ over $K$ vanishes. It is shown that there is an explicit infinite set of cubefree integers $m$ such that the $3$-Selmer groups over $K$ of $E_{m^2 a}$ and $E_{m^4 a}$ both vanish. In particular, the ranks of these cubic twists are seen to be $0$ over $K$. Our results are proven by studying stability properties of $3$-Selmer groups in cyclic cubic extensions of $K$, via local and global Galois cohomological techniques.
\end{abstract}

\section{Introduction}

\subsection{Motivation and historical background}
\par Given an elliptic curve $E$ over $\Q$, and a subgroup $H$ of $\op{Aut}_{\bar{\Q}}(E)$, it is natural to consider the twist family of elliptic curves $E_\zeta$ indexed by cohomology classes $\zeta\in H^1(\op{Gal}(\bar{\Q}/\Q), H)$. When $H=\{\pm 1\}$, the associated twists are called quadratic twists and are in bijection with $\Q^\times /(\Q^\times)^2$. Goldfeld's conjecture predicts that the average rank of quadratic twists $E_m$ over $\Q$ is $1/2$. In fact, the conjecture predicts that when $m$ is ordered according to its absolute value, exactly half of the quadratic twists have rank $0$ and the other half have rank $1$. The rank distribution in quadratic twist families is studied for instance in \cite{Brumer,HeathBrown, MazurRubin, smith,KrizLi, bhargava}. Set $K:=\Q(\mu_3)$ and let $\cO_K$ be its ring of integers. We consider certain families of elliptic curves with rational $3$-isogeny that have complex multiplication by $\cO_K$. For a sixth power free integer $a \neq 0$, set $E_a$ to be the elliptic curve \[E_a:y^2=x^3+a.\] Then for a cubefree integer $m$, we have the cubic twists $E_{m^2 a}$ and $E_{m^4 a}$ of $E_a$.
Let $\op{Sel}_3(E_a/K)$ (resp. $\Sh(E_a/K)$) denote the $3$-Selmer group (resp. Tate--Shafarevich group) of $E_a$ over $K$. The Selmer group sits in a short exact sequence
\[0\rightarrow E_a(K)/3 E_a(K)\rightarrow \op{Sel}_3(E_a/K)\rightarrow \Sh(E_a/K)[3]\rightarrow 0.\] It is thus easy to see that $\op{Sel}_3(E_a/K)=0$ if and only if \[E_a(K)[3]=0, \op{rank}(E_a(K))=0\text{ and }\Sh(E_a/K)[3]=0.\]
The main focus of this article is to study the following natural question.

\begin{question}
    Fix a sixth power free integer $a$ such that $\op{Sel}_3(E_a/K)=0$. What can one say about the family of cubefree integers $m$ for which $\op{Sel}_3(E_{m^2 a}/K)=0$ and $\op{Sel}_3(E_{m^4 a}/K)=0$? 
\end{question}

We construct explicit infinite families of numbers $m$ for which the above Selmer groups vanish.

\subsection{Main results} We state the main results of this article. It is shown that there is an explicit positive density of primes $\mathcal{Q}_a$ defined by local congruence conditions, such that if $m$ is divisible only by primes $\ell\in \mathcal{Q}_a$, then, $\op{Sel}_3(E_{m^2 a}/K)=0$ and $\op{Sel}_3(E_{m^4 a}/K)=0$. In particular, \[\op{rank} E_{m^2 a} (K),\op{rank} E_{m^4 a}(K)=0\text{ and }\Sh(E_{m^2 a}/K)[3], \Sh(E_{m^4 a}/K)[3]=0.\] When we specialize to the families $E_{\ell^2 a}$ and $E_{\ell^4 a}$, where $\ell$ is a prime number, we find that our result holds for an explicit positive density of primes $\ell$. 

 \par For a prime number $\ell$, set $\F_\ell:=\Z/\ell \Z$. If $E_a$ has good reduction at $\ell$, we set $\widetilde{E}_a(\F_\ell)$ to denote the group of $\F_\ell$-points on the reduced curve of $E_a$ at $\ell$. More generally, if $m=\ell_1\dots \ell_k$, for a fixed value of $k$ and $\ell_i$ prime, then there is a positive density of such numbers for which our result holds.
\begin{lthm}\label{main thm of paper}
     Let $a$ be an integer and $E_a$ be the elliptic curve defined by $y^2=x^3+a$. Assume that the following conditions are satisfied:
    \begin{enumerate}
        \item $a$ is not of the form $n^2$ or $-3 n^2$ for $n\in \Z$,
        \item $\op{Sel}_3(E_a/K)=0$.
    \end{enumerate} Set $\mathcal{Q}_a$ to denote the set of primes $\ell\nmid a$ such that the following conditions are satisfied:
    \begin{enumerate}
        \item $\ell\equiv 1\pmod{18}$, 
        \item $\ell$ is a cube modulo $q$ for each prime $q$ that divides $a$. 
          \item $\widetilde{E}_a(\F_\ell)[3]=0$.
    \end{enumerate}
    Then, the set $\mathcal{Q}_a$ is infinite with natural density given by
    \[\mathfrak{d}(\mathcal{Q}_a)=\frac{1}{8 \times 3^{s+1}},\] where $s$ is the number of prime divisors $q$ of $a$ such that $q\equiv 1\pmod{3}$. Moreover, if $m$ is any cubefree integer divisible only by primes in $\mathcal{Q}_a$, then, 
    \[\op{Sel}_3(E_{m^2 a}/K)=0 \text{ and }\op{Sel}_3(E_{m^4 a}/K)=0 .\]
\end{lthm}
For integers $a\in [-20, 20]$, it is shown that $\op{Sel}_3(E_a/K)=0$ for \[a=-17, -16,-14, -10, -9, -8, -6, -5, -1, 6,7,8, 13, 14, 20,\] and thus Theorem \ref{main thm of paper} applies to these values of $a$. The values of $\mathfrak{d}(\mathcal{Q}_a)$ are given in Table \ref{only table}. One of the main features of our result is that the set of primes $\mathcal{Q}_a$ can be made explicit. For instance, when $a=-1$, we have that 
\[\mathcal{Q}_{-1}\cap [1, 1000]=\{19,127, 163, 199, 271, 307, 379, 487, 523, 631, 739, 811, 883, 919, 991\}.\] For further details on numerical computations, we refer to sections \ref{illustrative example section} and \ref{numerical data section}.
\subsection{Methodology}
\par The study of Selmer ranks in twist families is related to the study the rank growth and stability of elliptic curves defined over a number field in a $\Z/p\Z$-extension for an odd prime $p$. Let $K$ be a number field containing the $p$-th roots of unity and $E$ be an elliptic curve over $K$ such that the $p$-Selmer group $\op{Sel}_p(E/K)=0$. Let $L/K$ be a cyclic $p$-extension. Then, it is shown that under certain conditions on $E$ and $L$, the Selmer group $\op{Sel}_p(E/L)=0$ as well. This strategy has been adapted from the technique of Mazur--Rubin \cite{MazurRubinAnnals, MazurRubinSilverberg}, Brau \cite{brau2014selmer}. In keeping with these works, we study stability questions in prime cyclic extensions for any prime number $p\geq 3$, cf. Theorem \ref{main theorem of section 3} and Theorem \ref{kummer extensions theorem} for further details.
\subsection{Related work} In recent years, there has been an increasing interest in rank distribution questions for families of elliptic curves from the perspective of \emph{arithmetic statistics}. These techniques involve the parameterization of Selmer groups of elliptic curves and abelian varieties by integral forms in many variables. The distribution properties of such integral forms can be studied via techniques from analytic number theory and the geometry of numbers, and has gained momentum in works of Bhargava--Shankar \cite{BS1,BS2,BS3}, Bhargava--Elkies--Shnidman \cite{bes}, Alp\"oge--Bhargava--Shnidman \cite{alpoge2022integers}, Shnidman--Weiss \cite{shnidman2021ranks, ShnidmanWeiss} etc. Our techniques on the other hand rely on local and global arguments in Galois cohomology (in keeping with those used by Mazur--Rubin, Brau), and one of the consequences is that we obtain very explicit descriptions for the twist families for which the Selmer group vanishes. Related questions have been studied by Satg\'e \cite{satge}, Iskra \cite{Iskra}, Goto \cite{Goto} and Bandini \cite{Bandini}. It is conceivable that our results should generalize to the context of abelian varieties with complex multiplication. For instance, Jacobians of twist families of bicyclic trigonal curves $y^3=f(x^2)$, cf. \cite{shnidman2021ranks} for details. 

\subsection{Organization} Including the introduction, the article consists of four sections. In section \ref{notation sec}, we set basic notation and recall definitions of Selmer groups and their properties. In section \ref{sec 3}, we study the stability of $p$-Selmer groups in $p$-cyclic extensions of number fields, leveraging the techniques of Mazur--Rubin, Brau and \v{C}esnavi\v{c}ius. We prove our main results for cubic twist families in section \ref{sec 4}. These results are backed up by explicit numerical computations in subsections \ref{illustrative example section} and \ref{numerical data section}.

\subsection{Acknowledgement}
\par Thanks to Antonio Lei for answering some of our questions, and for helpful comments. We would like to thank the referee for the careful review of the manuscript.

\subsection*{Data availability} No data was analyzed in proving the results in the article.
\subsection*{Conflict of Interest} There is no conflict of interest that the authors wish to report.
\section{Notation and preliminaries}\label{notation sec}
\par In this section, we introduce certain preliminary notions and set up relevant notation. Given an abelian group $M$, set $\# M$ to denote its size. Set $M[n]$ to denote its $n$-torsion subgroup. Given a prime number $p$, set $M[p^\infty]:=\bigcup_k M[p^k]$. The $p$-adic Tate module associated to $M$ is the inverse limit
\[\op{T}_p(M):=\varprojlim_n M[p^n]\] with respect to natural multiplication by $p$ maps $M[p^{k+1}]\rightarrow M[p^k]$. 
\par Fix a number field $K$ and denote by $\Sigma_K$ the set of all nonarchimedian primes $v$ of $K$. Given a prime $v$, we shall let $\ell$ be the rational prime such that $v|\ell$. Set $\F_\ell$ to denote the finite field with $\ell$ elements. For $v\in \Sigma_K$, set $K_v$ to denote the completion of $K$ at $v$. Take $\cO_{K_v}$ to be the valuation ring of $K_v$ and $\pi_v$ a choice of uniformizer of $\cO_{K_v}$. Let $k_v$ be the residue field $\cO_{K_v}/\pi_v$. Denote by $i_v: K\hookrightarrow K_v$ the natural inclusion of $K$ into $K_v$. Let $\bar{K}$ (resp. $\bar{K}_v$) be an algebraic closure of $K$ (resp. $K_v$) and $\op{G}_K$ (resp. $\op{G}_{K_v}$) denote the absolute Galois group of $K$ (resp. $K_v$). Let $\op{I}_v$ denote the inertia subgroup of $\op{G}_{K_v}$. At each prime $v\in \Sigma_K$, choose an embedding $\iota_v: \bar{K}\hookrightarrow \bar{K}_v$ that extends $i_v$. Note that $\iota_v$ induces an inclusion $\iota_v^*: \op{G}_{K_v}\hookrightarrow\op{G}_K$. Let $M$ be a $\op{G}_K$-module and $i\in \Z_{\geq 0}$. Take $H^i(K, M)$ (resp. $H^i(K_v, M)$) to denote the cohomology group $H^i(\op{G}_K, M)$ (resp. $H^i(\op{G}_{K_v}, M)$). The restriction map induced by $\iota_v^*$ is denoted 
\[\op{res}_v: H^i(K, M)\rightarrow H^i(K_v, M).\]
\par Let $E$ be an elliptic curve over $K$ and $n$ be a positive integer. Given a prime $v\in \Sigma_K$ at which $E$ has good reduction, set $\widetilde{E}$ to denote the reduction of $E$ at $v$. Take $\widetilde{E}(k_v)$ to denote the group of $k_v$-points on $\widetilde{E}$. Set $E[n]$ to denote the $n$-torsion subgroup of $E(\bar{K})$. For $\cF\in \{K, K_v\}$, we have the exact sequence of $\op{G}_{\cF}$-modules
\[0\rightarrow E[n]\rightarrow E\xrightarrow{\times n} E\rightarrow 0,\] from which we arrive at the following Kummer sequence
\begin{equation}\label{local kummer map}0\rightarrow E(\cF)/n E(\cF)\xrightarrow{\delta_{E, \cF}} H^1(\cF, E[n])\rightarrow H^1(\cF, E)[n]\rightarrow 0.\end{equation}
\begin{definition}\label{def of n Selmer group}Let $n$ be an odd natural number. With respect to notation above, the $n$-Selmer group of $E$ over $K$ is defined as follows
\[\op{Sel}_n(E/K):=\op{ker}\left\{H^1(K, E[n])\xrightarrow{\Phi} \prod_{v\in \Sigma_K} H^1(K_v, E)[n] \right\}.\] Here the map $\Phi$ is the product of maps $\Phi_v$, given by the composite
\[H^1(K, E[n])\xrightarrow{\op{res}_v} H^1(K_v, E[n])\rightarrow H^1(K_v, E)[n].\]
\end{definition}
Let $\Sigma$ be a finite set of primes that contains the primes of $K$ that lie above the primes that divide $n$ and the primes at which $E$ has bad reduction. Let $K_\Sigma$ be the maximal algebraic extension of $K$ in which the primes $\ell\notin \Sigma$ are unramified. Set $\op{G}_{K, \Sigma}$ to denote the Galois group $\op{Gal}(K_\Sigma/K)$. Given a $\op{G}_{K, \Sigma}$-module $M$, set $H^i(K_\Sigma/K, M)$ to denote $H^i(\op{G}_{K, \Sigma}, M)$. The Selmer group $\op{Sel}_n(E/K)$ can also be described by finitely many local conditions
\[\op{Sel}_n(E/K)=\op{ker}\left\{ H^1(K_\Sigma/K, E[n])\xrightarrow{\Phi_\Sigma} \bigoplus_{v\in \Sigma} H^1(K_v, E)[n]\right\},\]where the map $\Phi_\Sigma:=\bigoplus_{v\in \Sigma} \Phi_v$, cf. \cite[Chapter 1]{CoatesSujatha}. Given a prime $p$, the natural map $E[p^k]\hookrightarrow E[p^{k+1}]$ induces a map at the level of $p$-primary Selmer groups
\[\op{Sel}_{p^n}(E/K)\rightarrow \op{Sel}_{p^{n+1}}(E/K).\]
The $p$-primary Selmer group is then defined to be the direct limit
\[\op{Sel}_{p^\infty}(E/K):=\varinjlim_n \op{Sel}_{p^n} (E/K).\]
\par Let $\Sh(E/K)$ denote the Tate--Shafarevich group of $E$ over $K$. The Selmer group fits into a short exact sequence 
\begin{equation}\label{n Selmer exact sequence}0\rightarrow E(K)/n E(K)\xrightarrow{\delta_{E,n}} \op{Sel}_n(E/K)\rightarrow \Sh(E/K)[n]\rightarrow 0. \end{equation}
Given a prime $p$, passing to the direct limit for integers $n$ of the form $p^k$, we have the following short exact sequence
\begin{equation}\label{basic Selmer exact sequence}0\rightarrow E(K)\otimes_{\Z} \Q_p/\Z_p\rightarrow \op{Sel}_{p^\infty}(E/K)\rightarrow \Sh(E/K)[p^\infty]\rightarrow 0. \end{equation}
\par As is well known, the Tate--Shafarevich group $\Sh(E/K)$ is conjectured to be finite. If $\Sh(E/K)[p^\infty]$ is known to be finite, then from \eqref{basic Selmer exact sequence}, 
\[\op{corank}_{\Z_p} \left(\op{Sel}_{p^\infty}(E/K)\right)=\op{rank} E(K). \] 
\par For each integer $k\geq 1$, multiplication by $p$ induces a map of Galois modules $E[p^{n+1}]\xrightarrow{\times p} E[p^n]$. The \emph{compact Selmer group} is defined to be the inverse limit
\[\mathfrak{S}_p(E/K):=\varprojlim_n 
 \op{Sel}_{p^n} (E/K)\] with respect to these maps. From \eqref{n Selmer exact sequence}, we obtain the following exact sequence
 \begin{equation}\label{compact selmer exact sequence}
     0\rightarrow E(K)\otimes_{\Z} \Z_p\rightarrow \mathfrak{S}_p(E/K)\rightarrow \op{T}_p\left(\Sh(E/K)\right)\rightarrow 0.
 \end{equation}
If the $p$-primary Tate--Shafarevich group $\Sh(E/K)[p^\infty]$ is finite, then its Tate--module vanishes. Thus, if $\Sh(E/K)[p^\infty]$ is finite, then, 
\[\op{rank}_{\Z_p}\left(\mathfrak{S}_p(E/K)\right)=\op{rank} E(K).\]

It follows from \cite[Lemma 1.8]{CoatesSujatha} that there is a short exact sequence
\begin{equation}\label{comparison of compact and p-primary selmer groups}0\rightarrow E(K)[p^\infty]\rightarrow \mathfrak{S}_p(E/K)\rightarrow \op{T}_p\left(\op{Sel}_{p^\infty}(E/K)\right)\rightarrow 0,\end{equation} and thus, 
\begin{equation}\label{corank equal rank}\op{corank}_{\Z_p} \left(\op{Sel}_{p^\infty}(E/K)\right)=\op{rank}_{\Z_p}\left( \mathfrak{S}_p(E/K)\right).\end{equation} 

\par Note if $\op{Sel}_p(E/K)=0$, then, $\Sh(E/K)[3^\infty]=0$, $\op{rank} E(K)=0$ and $E(K)[3^\infty]=0$. Thus it follows from \eqref{comparison of compact and p-primary selmer groups} that $\mathfrak{S}_p(E/K)=0$. It then follows from the Cassels--Poitou--Tate exact sequence \cite[Ch. 1 (3)]{CoatesSujatha} that the map $\Phi_\Sigma$ is surjective, from which we obtain a short exact sequence
\begin{equation}\label{CPT equation}
   0\rightarrow  \op{Sel}_n(E/K)\rightarrow H^1(K_\Sigma/K, E[n])\xrightarrow{\Phi_\Sigma} \bigoplus_{v\in \Sigma} H^1(K_v, E)[n]\rightarrow 0.
\end{equation}

\section{Rank stability in prime cyclic extensions}\label{sec 3}
\par Let $E$ be an elliptic curve over a number field $K$ and $p$ be an odd prime number. Denote by $\Sigma_p$ the set of primes $v$ of $K$ that lie above $p$. In this section, we assume that $\op{Sel}_p(E/K)=0$. Note that this is equivalent to requiring that $E(K)$ has rank $0$, and both $E(K)$ and $\Sh(E/K)$ have no nontrivial $p$-torsion.
\subsection{Local contributions to the growth of Selmer groups}
\par We take $L/K$ to be a finite Galois extension with Galois group $G=\op{Gal}(L/K)$ isomorphic to $\Z/p\Z$. Let $\sigma$ be a generator of $G$. Set $S$ to be the set of primes $v\in \Sigma_K$ such that any one of the following conditions are satisfied:
\begin{itemize}
    \item $v\in \Sigma_p$, 
    \item $E$ has bad reduction at $v$, 
    \item $v$ is ramified in $L$.
\end{itemize}
Denote by $S(L)$ the primes $w\in \Sigma_L$ such that $w|v$ for some prime $v\in S$. Then there are exact sequences
\[\begin{split} & 0\rightarrow \op{Sel}_p(E/K)\rightarrow H^1(K_S/K, E[p])\xrightarrow{\Phi_K} \bigoplus_{v\in S} H^1(K_v, E)[p]\rightarrow 0,  \\ 
 & 0\rightarrow \op{Sel}_p(E/L)^G\rightarrow H^1(K_S/L, E[p])^G\xrightarrow{\Phi_L} \left(\bigoplus_{w\in S(L)} H^1(L_w, E)[p]\right)^G.  \\ \end{split}\]
From the Cassels-Poitou-Tate sequence (cf. \eqref{CPT equation} and the discussion preceding it), the map $\Phi_K$ is surjective. For $v\in S$, consider the restriction map
\[\gamma_v: H^1(K_v, E)[p]\rightarrow \left(\bigoplus_{w|v} H^1(L_w, E)[p]\right)^G,\] where $w$ runs through all primes of $L$ that lie above $v$. Let 
\[\gamma: \bigoplus_{v\in S} H^1(K_v, E)[p]\rightarrow \left(\bigoplus_{w\in S(L)} H^1(L_w, E)[p]\right)^G\]
denote the direct sum of maps $\bigoplus_{v\in S} \gamma_v$.
\par There is a natural restriction map
\[\alpha: \op{Sel}_p(E/K)\rightarrow \op{Sel}_p(E/L)^G\] which fits into the fundamental diagram
 \begin{equation}\label{fdiagram}
\begin{tikzcd}[column sep = small, row sep = large]
& 0=\op{Sel}_p(E/K) \arrow{r}\arrow{d}{\alpha} & H^1(K_S/K, E[p])\arrow{r}{\Phi_{K}} \arrow{d}{\beta} & \bigoplus_{v\in S} H^1(K_v, E)[p] \arrow{d}{\gamma}  \arrow{r} & 0\\
0\arrow{r} & \op{Sel}_p(E/L)^G\arrow{r} & H^1(L_S/L, E[p])^G\arrow{r}  &\left(\bigoplus_{w\in S(L)} H^1(L_w, E)[p]\right)^G,
\end{tikzcd}
\end{equation}

From the snake lemma one obtains an exact sequence 
\begin{equation}\label{snake lemma exact sequence}
    0=\op{ker}(\alpha)\rightarrow \op{ker}(\beta)\rightarrow \op{ker}(\gamma)\rightarrow \op{cok}(\alpha)\rightarrow \op{cok}(\beta)\rightarrow \op{cok}(\gamma)
\end{equation}
The following basic result will be useful in many of our arguments.
\begin{lemma}\label{lemma V^G=0 implies V=0}
    Let $V$ be a finite dimensional $\F_p$-vector space and $G$ be a finite $p$-group which acts on $V$. Then, the following implication holds
    \[V^G=0\Rightarrow V=0.\]
\end{lemma}
\begin{proof}
    We have that 
    \[\# V=\sum_i \# \mathcal{O}_i,\] where the sum ranges over $G$-orbits of $V$. Suppose that $V^G=0$, then $\{0\}$ is a $G$-orbit of $V$ which contains only $1$-element. Every other $G$-orbit has $p^{n_i}$ elements with $n_i>0$. This implies that $\# V \equiv 1\mod{p}$, hence, $V=0$. 
\end{proof}
\begin{proposition}\label{SelpL=0 propn}
    Let $K$ be a number field and $E$ be an elliptic curve over $K$ such that $\op{Sel}_p(E/K)=0$. Let $L/K$ be a cyclic $p$-extension with Galois group $G:=\op{Gal}(L/K)$. Assume moreover that for all $v\in S$, the map $\gamma_v$ is injective. Then, we have that $\op{Sel}_p(E/L)=0$. In particular, the rank of $E(L)$ is $0$. 
\end{proposition}
\begin{proof}
    Since it is assumed that $\op{Sel}_p(E/K)=0$, we find that $E(K)[p]=0$. Note that $E(K)[p]=E(L)[p]^G$. We deduce from Lemma \ref{lemma V^G=0 implies V=0} that $E(L)[p]=0$. The restriction map $\beta$ in \eqref{fdiagram} sits in an inflation-restriction sequence
    \[0\rightarrow H^1(L/K, E(L)[p])\rightarrow H^1(K_S/K, E[p])\xrightarrow{\beta} H^1(K_S/L, E[p])^G\rightarrow H^2(L/K, E(L)[p]).\] The groups $H^i(L/K, E(L)[p])=0$ and hence, $\beta$ is an isomorphism. The exact sequence \eqref{snake lemma exact sequence} then implies that \[\op{cok}(\alpha)\simeq \op{ker}(\gamma)\simeq \bigoplus_{v\in S} \op{ker}\gamma_v.\]
    Since it is assumed that $\gamma_v$ is injective for all $v\in S$, we deduce that $\op{cok}(\alpha)=0$. Since $\op{Sel}_p(E/K)=0$, it follows that $\op{cok}(\alpha)=\op{Sel}_p(E/L)^G=0$. It then follows from Lemma \ref{lemma V^G=0 implies V=0} that $\op{Sel}_p(E/L)=0$, and this concludes the proof of the result. 
\end{proof}

\par We analyze conditions for $\gamma_v$ to be injective for $v\in S$. Let $w|v$ be a chosen prime of $L$ and consider the restriction map 
\[\gamma_v': H^1(K_v, E)[p]\rightarrow H^1(L_w, E)[p]^{\op{Gal}(L_w/K_v)}.\] Since $G$ permutes the primes $w$ of $L$ that lie above $v$ transitively, it follows that 
\begin{equation}\label{kernels are the same}\op{ker}\gamma_v=\op{ker}\gamma_v'.\end{equation}
\begin{lemma}\label{gamma_v injective when v splits}
    Let $v\in S$ and $w|v$ be a prime of $L$. Then, we have that 
    \[\op{ker}\gamma_v\simeq H^1(L_w/K_v, E(L_w))[p].\]
In particular, if $v$ splits in $L$, then, $\gamma_v$ is injective.
\end{lemma}
\begin{proof}
    From the inflation-restriction sequence, we find that 
      \[\op{ker}\gamma_v'\simeq H^1(L_w/K_v, E(L_w))[p].\]
      Since $\op{ker}\gamma_v\simeq \op{ker}\gamma_v'$, the result follows.
\end{proof}

Since $\gamma_v$ is injective when $v$ splits in $L$, we consider two cases
\begin{enumerate}
    \item $v\nmid p$ and $v$ is inert or ramified in $L$, 
    \item $v| p$ and $v$ is inert or ramified in $L$.
\end{enumerate}
We note that in the above cases, there is a unique prime $w|v$ and hence, $\gamma_v'=\gamma_v$. 
\begin{lemma}\label{gamma_v is injective for v does not divide p}
    For $v\in S\setminus \Sigma_p$ assume that \emph{either} of the following conditions is satisfied
    \begin{enumerate}
        \item $v$ is a prime of bad reduction and is split in $L$,
        \item $v$ is a prime of good reduction, $v$ is ramified in $L$ and $\widetilde{E}(k_v)[p]=0$. 
    \end{enumerate}
    Then, $\gamma_v$ is injective.
\end{lemma}
We remark that a prime $v\in S\setminus \Sigma_p$ at which $E$ has good reduction is necessarily ramified in $L$.
\begin{proof}
    The result follows from \cite[Proposition 2.8]{pathakray2024rank}.
\end{proof}
\par Throughout the rest of this section, we shall study the cases when $v|p$. This analysis is based on unpublished work of Brau \cite{brau2014selmer}, results of Mazur--Rubin \cite{MazurRubinAnnals} and Mazur--Rubin--Silverberg \cite{MazurRubinSilverberg}. For a prime $v$ of a number field $F$ and a $\op{G}_{F_v}$-module $V$, set 
\[H^1_{\op{nr}}(F_v, V):=\op{image}\left\{\op{inf}: H^1(\op{G}_{F_v}/I_v, V^{\op{I}_v})\longrightarrow H^1(K_v, V)\right\}.\]

\begin{definition}\label{selmer structure def}
    Let $p$ be an odd prime and  $F$ be a number field. For a continuous $p$-primary $\op{G}_F$-module $V$, a \emph{Selmer structure} $\cF$ on $V$ is a collection of cohomology classes \[H^1_\cF(F_v, V)\subseteq H^1(F_v, V)\] for $v\in \Sigma_F$ such that for all but finitely many primes $v$, 
    \[H^1_{\cF}(F_v, V)=H^1_{\op{nr}}(F_v, V).\]
    Associated to $\cF$ is the Selmer group 
    \[H^1_{\cF}(F, V):=\op{ker}\left\{ H^1(F, V)\longrightarrow \prod_{v\in \Sigma_F} \left(\frac{H^1(F_v, V)}{H^1_{\cF}(F_v, V)}\right)\right\}.\]
\end{definition} 
\par Given two Selmer structures $\cF$ and $\cG$, set 
\[\begin{split}
    & H^1_{\cF+\cG}(F_v, V):=H^1_{\cF}(F_v, V)+H^1_{\cG}(F_v, V),\\
    & H^1_{\cF\cap \cG}(F_v, V):=H^1_{\cF}(F_v, V)\cap H^1_{\cG}(F_v, V).
\end{split}\]

We would like to draw attention to two Selmer conditions $\cF$ and $\mathcal{A}$ associated to the $\op{G}_K$-module $E[p]$. For each prime $v\in \Sigma_K$, set 
\[H^1_{\cF}(K_v, E[p]):=\op{image}\left\{ \frac{E(K_v)}{p E(K_v)}\xrightarrow{\delta_{E,K_v}} H^1(K_v, E[p])\right\},\] where $\delta_{E,K_v}$ is the local Kummer map \eqref{local kummer map}.
Observe that $\op{Sel}_p(E/K)=H^1_{\cF}(K, E[p])$. Let $L/K$ be a $G$-extension where $G\simeq \Z/p\Z$. 
Set $\op{Res}_K^L E$ to denote the Weil--restriction of scalars of $E$ from $L$ to $K$. For every $K$-algebra $X$, we have that 
\[(\op{Res}_K^L E)(X)=E(L\otimes_K X).\]Let $A$ denote the kernel of the norm map $\op{Res}_K^L E\xrightarrow{N_{L/K}} E$. Let $\cO$ be the ring of integers of $\Q(\mu_p)$ and $\p$ be the ideal of $\cO$ that lies above $p$. The natural inclusion $\Z[G]\hookrightarrow \op{End}_K(\op{Res}_K^L E)$ induces a natural map $\Z[G]\rightarrow \op{End}_K(A)$, which factors as
\[\Z[G]\twoheadrightarrow \cO\hookrightarrow \op{End}_K(A).\] There is a canonical $\op{G}_K$-isomorphism $A[\p]\xrightarrow{\sim} E[p]$, cf. \cite[Proposition 4.1]{MazurRubinAnnals}. Let $\pi$ be a generator of $\p$ and consider the composed map
\[\lambda_{A, v}: \frac{A(K_v)}{ \pi A(K_v)} \hookrightarrow H^1(K_v, A[\p])\xrightarrow{\sim} H^1(K_v, E[p]). \] Let $H^1_{\mathcal{A}}(K_v, E[p])$ be the image of this composed map. This defines $\mathcal{A}$, which is referred to as the \emph{twisted Selmer condition} with respect to $L/K$.

\begin{lemma}\label{v local F + A lemma}
    For every prime $v\in S$, we have that \[\op{ker}\gamma_v \simeq \frac{H^1_{\cF+\mathcal{A}}(K_v, E[p])}{H^1_{ \mathcal{A}}(K_v, E[p])}\simeq \frac{H^1_{\cF}(K_v, E[p])}{H^1_{ \cF\cap \mathcal{A}}(K_v, E[p])}.\]
\end{lemma}
\begin{proof}
    The result above is \cite[Lemma 5.1]{brau2014selmer}, and follows from Proposition 4.4 of \emph{loc. cit}.
\end{proof}

\begin{lemma}\label{gamma_v is injective for v|p}
    Let $v\in \Sigma_p$ be a prime at which $E$ has good ordinary reduction, and assume that $v$ is ramified in $L$. Then if $\widetilde{E}(k_v)[p]=0$, the map $\gamma_v$ is injective.  
\end{lemma}

\begin{proof}
    Lemma \ref{v local F + A lemma} asserts that 
    \begin{equation}\label{boring eq 1}\op{ker}\gamma_v \simeq \frac{H^1_{\cF}(K_v, E[p])}{H^1_{\cF\cap \mathcal{A}}(K_v, E[p])}.\end{equation}
    It follows from \cite[Proposition 5.2]{MazurRubinAnnals} that 
    \begin{equation}\label{boring eq 2}\frac{H^1_{\cF}(K_v, E[p])}{H^1_{\cF\cap \mathcal{A}}(K_v, E[p])}\simeq \frac{E(K_v)}{\left(E(K_v)\cap N_{L_w/K_v} E(L_w)\right)}.\end{equation} 
    Denote by $E_1(K_v)$ (resp. $E_1(L_w)$) the kernel of the reduction map $E(K_v)\rightarrow \widetilde{E}(k_v)$ (resp. $E(L_w)\rightarrow \widetilde{E}(k_w)$). Since $\widetilde{E}(k_v)[p]=0$ and $L/K$ is a $p$-extension, it follows that 
    $\widetilde{E}(k_w)[p]=0$ (cf. Lemma \ref{lemma V^G=0 implies V=0}). Therefore, we find that
    \[\begin{split}
        E(K_v)/p E(K_v)\simeq E_1(K_v)/p E_1(K_v),\\
        E(L_w)/p E(L_w)\simeq E_1(L_w)/p E_1(L_w).\\
    \end{split}\]
     Hence it is easy to see that 
    \[\frac{E(K_v)}{\left(E(K_v)\cap N_{L_w/K_v} E(L_w)\right)}\simeq \frac{E_1(K_v)}{\left(E_1(K_v)\cap N_{L_w/K_v} E_1(L_w)\right)}.\]
    
    From \cite[Proposition B.2]{MazurRubinAnnals} it follows that \[\frac{E_1(K_v)}{\left(E_1(K_v)\cap N_{L_w/K_v} E_1(L_w)\right)}=0.\] Combining this with \eqref{boring eq 1} and \eqref{boring eq 2}, we conclude that $\gamma_v$ is injective.
\end{proof}

\begin{lemma}
    Let $v\in \Sigma_p$ be a prime at which $E$ has good reduction and assume that $v$ is inert in $L$. Then, $\gamma_v$ is injective.
\end{lemma}
\begin{proof}
From the proof of Lemma \ref{gamma_v is injective for v|p}, we find that 
\[\op{ker}\gamma_v\simeq \frac{E_1(K_v)}{\left(E_1(K_v)\cap N_{L_w/K_v} E_1(L_w)\right)}.\]
The norm map $N_{L_w/K_v}: E_1(L_w)\rightarrow E_1(K_v)$ is surjective when $L_w/K_v$ is unramified (cf. \cite[Proposition 3.1]{hazewinkel}). Therefore, $\gamma_v$ is injective and this proves the result.
\end{proof}

\par Next, we deal with the case when $v\in \Sigma_p$ is a prime at which $E$ has good supersingular reduction. Let $\cK$ be a finite extension of $\Q_p$ with valuation ring $\cO_\cK$, maximal ideal $\mathfrak{m}_\cK$ and residue field $k=\cO_\cK/\mathfrak{m}_\cK$. Let $\cF(X, Y)\in \cO\llbracket X,Y\rrbracket$ be a formal group law and $\cF(\mathfrak{m}_\cK)$ be the induced abelian group. The group operation is defined as follows 
\[x\oplus_{\cF} y:=\cF(x,y)\] for $x,y \in \mathfrak{m}_\cK$. Let $\cL/\cK$ be a finite Galois extension with Galois group $\mathcal{G}=\op{Gal}(\cL/\cK)$. Writing $\cG=\{\sigma_1, \dots, \sigma_n\}$, consider the induced norm map
\[N_{\cL/\cK}: \cF(\mathfrak{m}_\cL)\rightarrow \cF(\mathfrak{m}_\cK),\] defined by 
\[N_{\cL/\cK}(x):=\sigma_1(x)\oplus_{\cF} \sigma_2(x)\oplus_{\cF}\dots \oplus_{\cF} \sigma_n(x).\]
Assume that $\cG\simeq \Z/p\Z$ and $\sigma$ be a generator of $\cG$. Let $v_\cL$ (resp. $v_\cK$) denote the normalized valuation of $\cO_\cL$ (resp. $\cO_\cK$). Assume that $\cL/\cK$ is a ramified extension. Set $\pi_\cL$ (resp. $\pi_\cK$) to denote a uniformizer of $\cO_\cL$ (resp. $\cO_\cK$). Setting
\begin{equation}\label{defn of t}t:=v_\cL\left(\sigma(\pi_\cL)-\pi_\cL\right)-1,\end{equation} we note that $t\geq 1$ since $\cL/\cK$ must be a wildly ramified extension. On the other hand, suppose $\cL/\cK$ is unramified, then, $t:=-1$. 

\par Suppose now that $L/K$ is a $\Z/p\Z$-extension of number fields, $v\in \Sigma_K$ and $w|v$ is a prime of $L$. Let $t_v:=t$ be the invariant defined above \eqref{defn of t} for $\cL:=L_w$ and $\cK:=K_v$. 

\begin{proposition}\label{supersingular redn gamma_v injective propn}
    Let $E$ be an elliptic curve with good supersingular reduction at a prime $v\in \Sigma_p$. Suppose that $v$ is ramified in $L$. Then, $\gamma_v$ is injective if and only if $t_v=1$.
\end{proposition}
\begin{proof}
    The result is \cite[Proposition 5.10]{brau2014selmer}; we give a brief sketch. Let $w$ be the prime of $L$ that lies above $v$. We identify $\op{Gal}(L_w/K_v)$ with $G$. Since $E$ has good supersingular reduction at $v$, we find that 
    \[a_v(E):=\op{Norm}_{K/\Q}(v)+1-\# \widetilde{E}(k_v)\] is divisible by $p$. This implies that $\widetilde{E}(k_v)[p]=0$, in other words, 
    \[\left(\widetilde{E}(k_w)[p]\right)^{G}=0.\]
    It follows from Lemma \ref{lemma V^G=0 implies V=0} that $\widetilde{E}(k_w)[p]=0$. Let $E_1$ denote the formal group of $E$. Consider the short exact sequence of $G$-modules
    \[0\rightarrow E_1(L_w)\rightarrow E(L_w)\rightarrow \widetilde{E}(k_w)\rightarrow 0\] and the associated long exact sequence in cohomology
    \[\widetilde{E}(k_v)\rightarrow H^1(G, E_1(L_w))\rightarrow H^1(G, E(L_w))\rightarrow H^1(G, \widetilde{E}(k_w)). \]
    Since $p\nmid \#\widetilde{E}(k_w)$, it follows from the restriction-corestriction sequence that \[H^1(G,\widetilde{E}(k_w))=0.\] On the other hand, $H^1(G, E_1(L_w))$ is a $p$-primary group, and $p\nmid \# \widetilde{E}(k_v)$. Hence, there is an isomorphism 
    \[H^1(G, E(L_w))\simeq H^1(G, E_1(L_w)).\]
    The proof of \cite[Proposition 3.6]{brau2014selmer} shows that the Herbrand-quotient of $E_1(L_w)$ is $1$, and hence, 
    \[H^1(G, E_1(L_w))\simeq H^2(G, E_1(L_w))\simeq E_1(K_v)/N_{L_w/K_v}\left( E_1(L_w)\right).\] Then \cite[Proposition 3.5]{brau2014selmer} shows that
    \[E_1(K_v)=N_{L_w/K_v}\left( E_1(L_w)\right)\Leftrightarrow t_v=1.\]
\end{proof}
\begin{lemma}\label{gamma_v injective base change lemma}
    Let $E$ be an elliptic curve and $v$ be a prime of $K$. Assume that there is a finite Galois extension $K'/K$ with $p\nmid [K':K]$ such that $\gamma_{v'}$ is injective for some prime $v'|v$ of $K'$. Then, $\gamma_v$ itself is injective.
\end{lemma}

\begin{proof}
    Assume without loss of generality that $v$ is inert or ramified in $L$. Setting $L':=K'\cdot L$, let $w'$ be the prime of $L'$ that lies above $v'$. Let $w$ be a prime of $L$ that lies below $w'$, as depicted
     
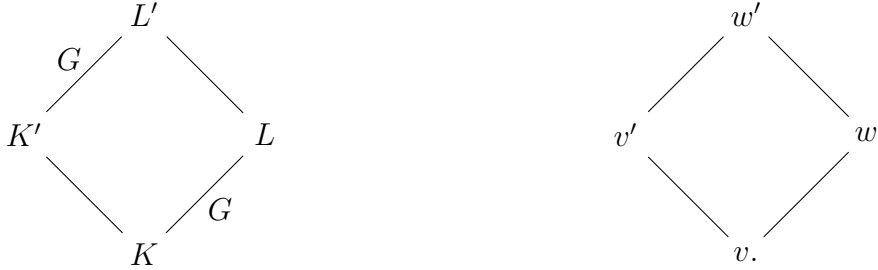
\[ \begin{tikzpicture}[scale=.8]
    \begin{scope}[xshift=0cm]
    \node (Q1) at (0,0) {$K$};
    \node (Q2) at (2,2) {$L$};
    \node (Q3) at (0,4) {$L'$};
    \node (Q4) at (-2,2) {$K'$};

    \draw (Q1)--(Q2) node [pos=0.7, below,inner sep=0.25cm] {$G$};
    \draw (Q1)--(Q4) node [pos=0.7, below,inner sep=0.25cm] {};
    \draw (Q3)--(Q4) node [pos=0.7, above,inner sep=0.25cm] {$G$};
    \draw (Q2)--(Q3) node [pos=0.3, above,inner sep=0.25cm] {};
    \end{scope}

    \begin{scope}[xshift=10cm]
    \node (Q1) at (0,0) {$v.$};
    \node (Q2) at (2,2) {$w$};
    \node (Q3) at (0,4) {$w'$};
    \node (Q4) at (-2,2) {$v'$};

    \draw (Q1)--(Q2) node [pos=0.7, below,inner sep=0.25cm] {};
    \draw (Q1)--(Q4) node [pos=0.7, below,inner sep=0.25cm]{};
    \draw (Q3)--(Q4) node [pos=0.7, above,inner sep=0.25cm]{};
    \draw (Q2)--(Q3) node [pos=0.7, above,inner sep=0.25cm]{};
    \end{scope}
    \end{tikzpicture}
\]Since $K'/K$ is a prime to $p$-extension,  $\op{Gal}(L'_{w'}/K'_{v'})$ is naturally isomorphic to $\op{Gal}(L_w/K_v)$. By hypothesis, the kernel of $\gamma_{v'}$ is trivial, hence, \[H^1(\op{Gal}(L_{w'}'/K_{v'}'), E)[p]=0.\]Consider the inflation-restriction sequence for the field extensions $L'_{w'}/K'_{v'}/K_v$
\[0\rightarrow H^1(K'_{v'}/K_v, E)[p]\rightarrow H^1(L'_{w'}/K_v, E)[p]\rightarrow H^1(L'_{w'}/K'_{v'}, E)[p].\] Since $K'/K$ is a prime to $p$-extension, it follows that 
\[H^1(K'_{v'}/K_v, E)[p]=0,\] and we deduce that
    \[H^1(\op{Gal}(L_{w'}'/K_{v}), E)[p]=0. \] Once again, from the inflation-restriction sequence for the fields $L'_{w'}/L_{w}/K_v$, we deduce that 
    \[H^1(\op{Gal}(L_w/K_v), E)[p]=0\] and hence, $\gamma_v$ is injective. 
\end{proof}

\begin{lemma}\label{potential v|p lemma}
    Let $E$ be an elliptic curve and $v|p$ be a prime of $K$ at which $E$ has potentially good reduction. Assume that there is a finite Galois extension $K'/K$ such that $p\nmid [K':K]$ and for a prime $v'|v$ of $K'$, either of the following conditions hold
    \begin{enumerate}
        \item $E$ has good ordinary reduction at $v'$ and $\widetilde{E}(k_{v'})[p]=0$, or, 
        \item $E$ has good supersingular reduction at $v'$ and $t_{v'}\leq 1$. 
    \end{enumerate}
    Then, $\gamma_v$ is injective.
\end{lemma}

\begin{proof}
    The result follows from Lemma \ref{gamma_v is injective for v|p}, Proposition \ref{supersingular redn gamma_v injective propn} and Lemma \ref{gamma_v injective base change lemma}.
\end{proof}

\subsection{Rank stability in prime cyclic extensions}
\par In this subsection, we study applications of our local computations to rank stability questions in $\Z/p\Z$-extensions of $K$. 
\begin{theorem}\label{main theorem of section 3}
    Let $K$ be a number field and $E$ be an elliptic curve over $K$. Let $p$ be an odd prime number and assume that $\op{Sel}_p(E/K)=0$. Let $L/K$ be a degree $p$ Galois extension and $S$ be the set of primes $v$ of $K$ such that any one of the following hold:
\begin{itemize}
    \item $v|p$, 
    \item $E$ has bad reduction at $v$, 
    \item $v$ is ramified in $L$.
\end{itemize}
Write $S=\Sigma_1\cup \Sigma_2 \cup \Sigma_p$, where $\Sigma_1$ is the set of primes of $K$ at which $E$ has bad reduction and $\Sigma_2:=S\setminus (\Sigma_1 \cup \Sigma_p)$. Assume that the following conditions are satisfied.
\begin{enumerate}
    \item Each prime $v\in \Sigma_1$ splits completely in $L$.
    \item For each prime $v\in \Sigma_2$, we have that $\widetilde{E}(k_v)[p]=0$.
    \item For each prime $v\in \Sigma_p$ which is not completely split in $L$, there is a prime to $p$ Galois extension $K'/K$ and a prime $v'|v$ of $K'$ such that either 
    \begin{itemize}
        \item $E$ has good ordinary reduction at $v'$ and $\widetilde{E}(k_{v'})[p]=0$, or, 
        \item $E$ has good supersingular reduction at $v'$ and $t_{v'}\leq 1$. 
    \end{itemize}
\end{enumerate}
Then, the $p$-Selmer group $\op{Sel}_p(E/L)=0$. In particular, $\op{rank} E(L)=0$ and $\Sh(E/L)[p]=0$.
\end{theorem}
\begin{proof}
    The Proposition \ref{SelpL=0 propn} asserts that if $\gamma_v$ is injective for all primes $v\in S$, then, $\op{Sel}_p(E/L)=0$. For each prime $v\in \Sigma_1$, it follows from the Lemma \ref{gamma_v injective when v splits} that the map $\gamma_v$ is injective. For each prime $v\in \Sigma_2$ such that $v\nmid p$, it follows from Lemma \ref{gamma_v is injective for v does not divide p} that $\gamma_v$ is injective. On the other hand, for $v|p$, it follows from Lemma \ref{potential v|p lemma} that $\gamma_v$ is injective.
\end{proof}
Assuming that $K$ contains $\mu_p$, we shall consider Kummer extensions $L=K(\alpha^{\frac{1}{p}})$ for $\alpha\in K^*/(K^*)^p$. The following is a well known criterion for the ramification and splitting behavior of primes in such Kummer extensions. Given a prime $v\in \Sigma_K$, let $v(\alpha)$ denote $v_{K_v}(\alpha)$ by abuse of notation.

\begin{proposition}\label{ramification in kummer extensions Gras prop}
    Let $L/K$ be as above and $v\in \Sigma_K$ and $i_v: K\hookrightarrow K_v$ denote the natural embedding. Then the following assertions hold. 
    \begin{enumerate}
        \item Assume that $v\nmid p$. Then, $v$ is unramified in $L$ if and only if $v(\alpha)\equiv 0\pmod{p} $.
        \item Assume that $v|p$. Then, $v$ is unramified in $L$ if and only if the congruence
        \[\frac{\alpha}{x_v^p}\equiv 1\pmod{\pi_v^{p e_v}}\] has a solution for $x_v\in K_v^\times$, where $e_v$ is the ramification index of $v$ in $K/\Q(\mu_p)$.
        \item The prime $v$ splits in $L$ if and only if $i_v(\alpha)\in (K_v^\times)^p$.
    \end{enumerate}
    
\end{proposition}
\begin{proof}
    The above result is \cite[Theorem 6.3]{GrasCFT}. 
\end{proof}
Let us consider the special case where $K:=\Q(\mu_p)$.
\begin{lemma}\label{t <=1 lemma}
    Suppose that $K=\Q(\mu_p)$ and let $v$ be the unique prime of $K$ that lies above $p$. Let $m$ be a rational integer such that $m\notin (K^\times)^p$. Setting $L:=K(m^{\frac{1}{p}})$, we make the following assumptions. 
    \begin{enumerate}
         \item Suppose that $m^{p-1}\equiv 1\mod{p^2}$. Then, $v$ splits in $L$.
        \item Suppose that $p\nmid m$, then, $t_v\leq 1$. 
    \end{enumerate}
\end{lemma}
\begin{proof}
    The condition that $m^{p-1}\equiv 1\mod{p^2}$ implies that $m\in (\Q_p^\times)^p$. In particular, this implies that $m\in (K_v^\times)^p$ and thus $v$ splits in $L$ by Proposition \ref{ramification in kummer extensions Gras prop}. This proves part (1). For Part (2), assume without loss of generality that $p||(m^{p-1}-1)$. The result then follows from \cite[Theorem 5.6]{Viviani}. 
\end{proof}

\begin{theorem}\label{kummer extensions theorem}
    Let $p$ be an odd prime number and $K=\Q(\mu_p)$. Assume that $\op{Sel}_p(E/K)=0$ and $m$ be a rational integer such that $m\notin (K^\times)^p$. Let $Q_1$ be the set of primes $v\nmid p$ at which $E$ has good reduction and $\widetilde{E}(k_v)[p]=0$. Let $Q_2$ be the set of primes $v\nmid p$ at which $E$ has bad reduction. Then, assume that the following conditions are satisfied.
    \begin{enumerate}
        \item The only primes $v\nmid p$ for which $v(m)\not\equiv 0\mod{p}$ are contained in $Q_1$.
        \item For each prime $v\in Q_2$, we have that $i_v(m)\in (K_v^\times)^p$. 
        \item For each prime $v \mid p$ at which $E$ has good ordinary reduction and $\widetilde{E}(k_v)[p]\neq 0$, we have that $i_v(m)\in (K_v^\times)^p$. 
        \item For each prime $v\mid p$ at which $E$ has supersingular reduction, we have that $v(m)\equiv 0\mod{p}$.
        \item For each prime $v\mid p$ of bad reduction for $E$, we have that $i_v(m)\in (K_v^\times)^p$.
    \end{enumerate}
    Then, we have that $\op{Sel}_p(E/K(m^{1/p}))=0$. 
\end{theorem}
\begin{proof}
    Setting $L:=K(m^{1/p})$ it suffices to show that the conditions of Theorem \ref{main theorem of section 3} are satisfied. This is a direct consequence of Proposition \ref{ramification in kummer extensions Gras prop} and Lemma \ref{t <=1 lemma}.
\end{proof}

\section{Elliptic curves with rational $3$-isogeny : Stability results}\label{sec 4}
\par We illustrate our results from the previous section to the case when $p=3$. In this setting, we may consider families of elliptic curves $E_{/\Q}$ that have a rational $3$-isogeny. More specifically, we consider the family of elliptic curves given by $E:=E_a: y^2=x^3+a$, where $a$ is a cubefree integer. 
Let $\varphi_a: E\rightarrow E'$ be a rational $3$-isogeny, where $E':y^2=x^3-27a$. Set $\zeta:=\frac{-1+\sqrt{-3}}{2}$ and $\varpi:=1-\zeta$. We have that $\varpi=-\sqrt{-3}\zeta$, and hence, $\varpi^6=-27$. Thus, over $K:=\Q(\mu_3)$, the curves $E'$ and $E$ are isomorphic via $\theta_\varpi: E'\xrightarrow{\sim} E$, defined by \[\theta_\varpi(x,y) :=\left( \frac{x}{\varpi^2},  \frac{y}{\varpi^3}\right).\] Hence, $E$ is $3$-isogenous to itself over $K$. Let $\varphi:=\theta_{\varpi}\circ \varphi_a:E\rightarrow E$ be the isogeny which is explicitly given as follows \[\varphi(x,y)=\left( \frac{x^3+4a}{\varpi^2x^2}, \frac{y(x^3-8a)}{\varpi^3x^3} \right).\]
The kernel of $\varphi$ is the order $3$ subgroup $\{O, (0,\sqrt{a}), (0, -\sqrt{a})\} \subset E(\bar{K})$, cf. \cite[p. 3]{top}. Denote by $E(K)[\varphi]$ the kernel of the map $\varphi: E(K)\rightarrow E(K)$. Since $\varphi$ is self-dual, we find that $E(K)[\varphi^2]=E(K)[3]$, and thus, $E(K)[\varphi]$ is contained in $E(K)[3]$.

\begin{lemma}
    With respect to notation above, the following are equivalent
    \begin{enumerate}
        \item $a \notin (K^\times)^2$,
        \item $E(K)[\varphi]=0$,
        \item $E(K)[3]=0$.
    \end{enumerate}
\end{lemma}
\begin{proof}
    \par Note that the point $(0, \sqrt{a})$ of $E(\bar{K})$ has order $3$ and is the generator of $E[\varphi]$. Thus, $a$ is not a square in $K^\times$ if and only if $(0, \sqrt{a})$ is not contained in $E(K)$. This shows that 
    \[a\notin (K^\times)^2\Leftrightarrow E(K)[\varphi]=0,\]i.e., conditions (1) and (2) are equivalent.
    
\par There is a natural short exact sequence \[0 \rightarrow E[\varphi] \rightarrow E[3] \xrightarrow{\varphi} E[\varphi] \rightarrow 0.\]
    Since $\varphi$ is defined over $K$, taking Galois invariants gives the exact sequence
    \[0 \rightarrow E(K)[\varphi] \rightarrow E(K)[3] \xrightarrow{\varphi} E(K)[\varphi].\]
    Thus, one has $E(K)[\varphi]=0$ if and only if $E(K)[3]=0$. This proves that conditions (2) and (3) are equivalent, completing the proof. 
\end{proof}
\subsection{Results of Mazur and Rubin on the parity of ranks}\label{sec 4.1}
\par Note that $E$ is isogenous to its quadratic twist by $-3$ over $\Q$ and hence $\op{rank} E(K)=2\op{rank} E(\Q)$. In particular, the rank of $E(K)$ is even. It is however still natural to study the parity of the rank of $E(K)$ as an $\cO_K$-module. We briefly discuss parity results of Mazur and Rubin \cite{MazurRubinAnnals} for cubic twist families.

\par Let $m$ be a cubefree integer. Consider the cubic twists $E_{(m)}$ and $E_{(m^2)}$ of $E$, defined as follows
\[\begin{split}
    E_{(m)}:=E_{m^2a}:y^2=x^3+m^2 a;\\
    E_{(m^2)}:=E_{m^4a}:y^2=x^3+m^4 a.
\end{split}\]
In the next subsection, it is proven that for many values of $m$, the groups $E_{(m)}(K)$ and $E_{(m^2)}(K)$ both have rank $0$. Setting $L:=K(\sqrt[3]{m})$, we take $\chi=\chi_m:\op{Gal}(L/K)\rightarrow \mu_3$ to be the cubic character of order $3$ defined as follows
\[\chi(\sigma):=\frac{\sigma (\sqrt[3]{m})}{\sqrt[3]{m}},\] for $\sigma\in \op{Gal}(L/K)$. Since $\op{Aut}_{\bar{K}}(E)\simeq \cO_K^\times=\mu_6$, we find that the action of $\op{G}_K$ on $\op{Aut}_{\bar{K}}(E)$ is trivial. We view $\mu_3$ as a subgroup of $\op{Aut}_{\bar{K}}(E)$ and thus view $\chi$ as a cohomology class as follows \[\chi\in \op{Hom}\left(\op{G}_K, \op{Aut}_{\bar{K}}(E)\right)=H^1\left(\op{G}_K, \op{Aut}_{\bar{K}}(E)\right).\]
The elliptic curve $E_{(m)}$ (resp. $E_{(m^2)}$) is the twist $E_\chi$ (resp. $E_{\chi^2}$). In other words, there is an isomorphism $\phi: E_{\chi}\xrightarrow{\sim} E$ over $\bar{K}$ such that $\phi^\sigma=\chi(\sigma) \phi$, cf. \cite[section X.2]{silverman}. For $P\in E_{\chi}(\bar{K})$, we have that
\begin{equation}\label{relation Echi}\sigma(\phi(P))=\chi(\sigma)\phi(\sigma(P)).\end{equation}
Given a $\op{G}_K$-module $M$, take 
\[M^\chi:=\{m\in M\mid \sigma(m)=\chi(\sigma) m\text{ for all }\sigma\in \op{G}_K\}.\]In view of \eqref{relation Echi}, we have the identification $E_\chi(K)=E(L)^\chi$, and therefore, \[E(L) \simeq E(K)\oplus E(L)^\chi\oplus E(L)^{\chi^2}\simeq E(K)\oplus E_\chi (K)\oplus E_{\chi^2}(K).\]
From \eqref{compact selmer exact sequence}, one obtains the short exact sequence
\[0\rightarrow E_\chi(K)\otimes \Z_3\rightarrow \mathfrak{S}_3(E/L)^\chi\rightarrow \op{T}_3\left(\Sh(E/L)[3^\infty]\right)^\chi\rightarrow 0,\] and thus,
\[\op{rank} \left(\mathfrak{S}_3(E/L)^\chi\right)=\op{rank} E_\chi(K)\] provided $\Sh(E/L)[3^\infty]$ is finite.

Let $S$ be the set of primes of $K$ consisting of the prime above $3$, the primes at which $E$ has bad reduction and the primes that ramify in $L$. For each prime $v\in S$, there is a locally defined invariant $\delta_v\in \Z/2\Z$, cf. \cite[Definition 4.5]{MazurRubinAnnals}. Let $\Z_3[\chi]$ be the extension of $\Z_3$ generated by the values of $\chi$. The following result follows from Theorem A of \emph{loc. cit.} 

\begin{theorem}[Mazur--Rubin]\label{thm 5.1}
    With respect to notation above, we have that 
    \[\op{rank}_{\Z_3}\left( \mathfrak{S}_3(E/K)\right)-\op{rank}_{\Z_3[\chi]}\left( \mathfrak{S}_3(E/L)^\chi\right)\equiv \sum_{v\in S} \delta_v\pmod{2}. \]
\end{theorem}

Note that the value of $\op{rank}_{\Z_3} \left( \mathfrak{S}_3(E/K)\right)$ is expected to be even, and thus the above result predicts that 
\begin{equation}
    \op{rank}E_{(m)}(K), \op{rank}E_{(m^2)}(K)\equiv 2\left(\sum_{v\in S} \delta_v\right)\pmod{4}.
\end{equation}
Since $\op{rank} E_{(m)}(K)=2\op{rank} E_{(m)}(\Q)$, the above relation implies that 
\[\op{rank}E_{(m)}(\Q), \op{rank}E_{(m^2)}(\Q)\equiv \left(\sum_{v\in S} \delta_v\right)\pmod{2}.\]

\subsection{Main results}\label{sec 4.2}
In this section, we prove the main results of this article. As in the previous subsection, let $m$ be a cubefree integer. Assume now that $m$ is coprime to $3$. Set $L_m:=K(\sqrt[3]{m})$, note that $L_m/K$ is a cyclic cubic extension. Notice that the substitution 
\[(x,y)\mapsto (m^{2/3}x, m y )\]gives us that $E$ and $E_{(m)}$ (resp. $E_{(m)}$ and $E_{(m^2)}$) are isomorphic over $\Q(\sqrt[3]{m})$. We find there are isomorphisms of the Selmer groups
\[\op{Sel}_3(E/L_m)\simeq \op{Sel}_3(E_{(m)}/L_m)\simeq \op{Sel}_3(E_{(m^2)}/L_m).\]
Thus in particular, if $\op{Sel}_3(E/L_m)=0$, then $\op{rank} E_{(m)}(K)=0$ and $\op{rank} E_{(m^2)}(K)=0$. 
Let $\p:=\varpi \cO_K$ be the unique prime of $K$ that lies above $(3)$. Note that $E$, $E_{(m)}$ and $E_{(m^2)}$ all have additive reduction at $\p$.

From here on in, denote by $\ell$ a prime divisor of $m$, and $q$ a prime divisor of $a$.
\begin{definition}
    Let $\mathcal{M}_a$ be the set of rational primes $\ell$ such that the following conditions are satisfied
    \begin{itemize}
        \item $\ell\nmid a$, 
        \item $\ell\equiv 1\pmod{6}$, 
        \item $\widetilde{E}(\F_\ell)[3]=0$.
    \end{itemize}
\end{definition}

\begin{theorem}\label{section 4 basic thm}
    With respect to notation above, assume that 
    \begin{enumerate}
        \item $m\equiv 1 \pmod{9}$, 
        \item $\op{Sel}_3(E/K)=0$,
        \item $m$ is only divisible by primes of $\mathcal{M}_a$,
        \item $m$ is a cube modulo $q$ for each prime $q|a$.
    \end{enumerate}
    Then, $\op{Sel}_3(E/L_m)=0$, in particular, $\op{rank}E_{(m)}(K)=0$ and $\op{rank} E_{(m^2)}(K)=0$.
\end{theorem}
\begin{proof}
    We show that the conditions of Theorem \ref{main theorem of section 3} are indeed satisfied for $L=L_m$. 
    \begin{enumerate}
        \item First, we consider the condition for the primes of $\Sigma_1$. Recall that these are primes $v\nmid 3$ at which $E$ has bad reduction. We must show that such $v$ must be completely split in $L$. By part (3) of Proposition \ref{ramification in kummer extensions Gras prop}, we have that $v$ splits in $L$ if and only if $m\in (K_v^\times)^3$. Let $q$ be the rational prime such that $v|q$. Since $v$ is a prime of bad reduction, it follows that $q|2a$. Consider two cases. First, suppose that $q|a$. Note that $q\nmid m$ and $q\neq 3$. Thus, $m\in (\Z_q^\times)^3$ if and only if $m$ is a cube modulo $q$. Since $m$ is a cube modulo $q$ by assumption, we conclude that $m\in (\Z_q^\times )^3$ and hence, $m\in (K_v^\times)^3$. Thus $v$ is completely split in $L_m$ for all primes $v\in \Sigma_1$. Now consider the other case, namely, $q=2$. In this case, since $m$ is odd, $m\in \Z_2^\times$. However, it is easy to see that $\Z_2^\times=(\Z_2^\times)^3$, and thus, in particular $m\in (K_v^\times)^3$ and thus, $v$ splits in $L$.
        \item Recall that $\Sigma_2$ is the set of primes $v\nmid 3$ of $K$ at which $E$ has good reduction which are ramified in $L$. The condition (2) of Theorem \ref{main theorem of section 3} requires that $\widetilde{E}(k_v)[3]=0$ for all primes $v\in \Sigma_2$. This follows from the condition that $m$ is only divisible by primes of $\mathcal{M}_a$. In greater detail, let $v\in \Sigma_2$ and $\ell$ be the rational prime such that $v|\ell$. Since $v$ is ramified in $L$, it follows that $\ell|m$. Thus, $\ell\in \mathcal{M}_a$, and by definition, $\widetilde{E}(\F_\ell)[3]=0$. We are to show that $\widetilde{E}(k_v)[3]=0$. Since $\ell\in \mathcal{M}_a$, it follows by definition that $\ell\equiv 1\mod{3}$, thus $\ell$ splits in $K$. Thus, $k_v=\F_\ell$, and therefore, $\widetilde{E}(k_v)[3]=0$.
        \item Since $m\equiv 1\mod{9}$, it follows that $m\in (\Z_3^\times)^3$. Thus in particular, it follows that $m\in (K_v^\times)^3$ and thus $v$ splits in $L$ for the prime $v$ of $K$ that lies above $3$. Therefore, the condition (3) of Theorem \ref{main theorem of section 3} is satisfied. 
    \end{enumerate}
    The result thus follows from Theorem \ref{main theorem of section 3}.
\end{proof}

\begin{definition}\label{def of Q_a}
    Let $\mathcal{Q}_a$ be the set of primes $\ell$ such that 
    \begin{enumerate}
       \item $\ell\nmid a$,
        \item $\ell\equiv 1\pmod{18}$, 
        \item $\ell$ is a cube modulo $q$ for each prime $q$ that divides $a$. 
          \item $\widetilde{E}(\F_\ell)[3]=0$.
    \end{enumerate}

\end{definition}

\begin{corollary}\label{Q_a corollary}
    Let $K=\Q(\mu_3)$ and assume that $\op{Sel}_3(E/K)=0$. Let $m$ be an integer that is divisible only by primes of $\mathcal{Q}_a$. Then, we have that $\op{rank}E_{(m)}(K)=0$ and $\op{rank} E_{(m^2)}(K)=0$. 
\end{corollary}

\begin{proof}
    It is clear that the conditions of Theorem \ref{section 4 basic thm} are satisfied for integers $m$ that are divisible exclusively by primes in $\mathcal{Q}_a$.
\end{proof}

In light of Corollary \ref{Q_a corollary}, it is natural to investigate the density of the set of primes $\mathcal{Q}_a$. Let $\rho_{E,3}: \op{G}_\Q\rightarrow \op{GL}_2(\F_3)$ denote the Galois representation on $E[3]$. Set $\Q(E[3])$ to denote the Galois extension of $\Q$ which is \emph{cut out by} $E[3]$. In greater detail, $\Q(E[3])$ is defined as $\bar{\Q}^{\op{ker}\rho_{E,3}}$. It is clear that 
\[\op{Gal}(\Q(E[3])/\Q)\simeq \frac{\op{G}_\Q}{\op{ker}\rho_{E,3}}\simeq \op{image}\rho_{E,3}.\]
Given $\sigma\in \op{Gal}(\Q(E[3])/\Q)$, set $\rho(\sigma):=\rho_{E,3}(\sigma)$. Note that $E$ admits a rational $3$-isogeny $\varphi: E\rightarrow E'$, and therefore, there is a $\op{G}_{\Q}$-stable subgroup $C\subset E[3]$ that has order $3$. This means that the representation $\rho_{E,3}$ is reducible. Choose an $\F_3$-basis of $E[3]$ such that the image of $\rho_{E,3}$ is contained $B$, the group of upper triangular matrices $\mtx{\ast}{\ast}{0}{\ast}$ in $\op{GL}_2(\F_3)$. In other words, let $v_1$ be a generator of $C$ and let $v_2\in E[3]$ be another vector not contained in $C$. Then, with respect to the basis $\{v_1, v_2\}$, one obtains an upper triangular representation. Thus, there are characters $\varphi, \psi: \op{G}_\Q\rightarrow \F_3^\times=\{\pm 1\}$ such that 
\[\rho_{E, 3}=\mtx{\varphi}{\ast}{0}{\psi}.\] As a Galois module $C\simeq \F_3(\varphi)$. Let $\omega$ denote the mod-$3$ cyclotomic character and note that $\det \rho_{E,3}=\varphi \psi=\omega$. Write $B=T\ltimes U$, where $T$ consists of diagonal matrices $\mtx{\ast}{}{}{\ast}$ and $U$ consists of unipotent matrices $\mtx{1}{\ast}{0}{1}$. 

\begin{lemma}\label{E[3] image lemma}
    With respect to notation above, the following conditions are equivalent 
    \begin{enumerate}
        \item $E(K)[3]\neq 0$.
        \item We have that $\{\varphi, \psi\}=\{1, \omega\}$,
        \item the image of $\rho_{E,3}$ does not contain $-\op{Id}=\mtx{-1}{0}{0}{-1}$.
    \end{enumerate}
\end{lemma}

\begin{proof}
    \par First, we prove (1) is equivalent to (2). Since $\varphi$ and $\psi$ take values in $\{\pm 1\}$, we have that $\varphi^2=\psi^2=1$. Since $\varphi \psi= \omega$, we find that $\psi= \varphi^{-1} \omega=\varphi \omega$. Since $\omega_{|\op{G}_K}=1$, we find that $\varphi_{|\op{G}_K}=\psi_{|\op{G}_K}$. The condition that $E(K)[3]\neq 0$ is equivalent to the condition that $\varphi_{|\op{G}_K}=\psi_{|\op{G}_K}=1$. Note that any character $\epsilon:\op{G}_\Q\rightarrow \F_3^\times $ such that $\epsilon_{|\op{G}_K}=1$ is either $1$ or $\omega$. Thus, the condition (1) implies that $\varphi, \psi\in \{1, \omega\}$. Since $\varphi= \psi \omega$, we find that $\{\varphi, \psi\}=\{1, \omega\}$. Conversely, if $\{\varphi, \psi\}=\{1, \omega\}$, then, $\varphi_{|\op{G}_K}=\psi_{|\op{G}_K}=1$, and thus, $E(K)[3]\neq 0$. This implies that (2) implies (1). Thus, (1) and (2) are equivalent.

    \par Suppose that $\{\varphi, \psi\}=\{1, \omega\}$, then, $-\op{Id}\notin \op{image} \rho_{E,3}$ since either $\varphi$ or $\psi$ is $1$. Conversely, suppose that $\{\varphi, \psi\}\neq \{1, \omega\}$, then, $\varphi\neq 1$. Thus, there exists $\sigma\in \op{G}_K$ such that $\varphi(\sigma)=-1$. Since $\varphi(\sigma)\psi(\sigma)=\omega(\sigma)=1$, we find that $\psi(\sigma)=-1$ and hence $\rho_{E,3}(\sigma)^3=-\op{Id}$. This proves that (2) and (3) are equivalent.  
\end{proof}
Let $q$ be a prime dividing $a$. Suppose that $q\equiv 1\mod{3}$, then set $\mathbb{L}_q$ to denote the unique $\Z/3\Z$ extension of $\Q$ contained in $\Q(\mu_q)$. On the other hand, if $q\not\equiv 1\mod{3}$, set $\mathbb{L}_q:=\Q$. Denote by \[\mathbb{L}:=\left(\prod_{q|a} \mathbb{L}_q\right)\cdot \Q(\mu_{18})\] the compositum of all the fields $\mathbb{L}_q$ as $q$ ranges over the primes dividing $a$, and the field $\Q(\mu_{18})$. 

\begin{lemma}\label{dumb lemma 2}
    Let $\ell\nmid 6 a$ be a prime. Then, the conditions (1)--(3) of Definition \ref{def of Q_a} are satisfied by $\ell$ if and only if $\ell$ splits completely in $\mathbb{L}$. 
\end{lemma}
\begin{proof}
    A prime $\ell$ splits in $\Q(\mu_{18})$ if and only if $\ell \equiv 1\mod{18}$. Note that $\ell$ splits in $\mathbb{L}_q$ if and only if $\ell$ is a cube modulo $q$. This proves the result.
\end{proof}
Note that the field $\Q(\mu_{18})$ is unramified at all primes other than $2$ and $3$. On the other hand, $\mathbb{L}_q$ is totally ramified at $q$, and unramified at all other primes. Recall that $\mathbb{L}_q:=\Q$ unless $q\equiv 1\mod{3}$. This means that the fields $\{\mathbb{L}_q\}$ and $\Q(\mu_{18})$ are all mutually disjoint. Therefore, we find that \[[\mathbb{L}:\Q]=\prod_{q|a} [\mathbb{L}_q: \Q]\times [\Q(\mu_{18}):\Q]=3^s \times 6= 2 \times 3^{s+1},\] where $s$ is the number of primes $q\equiv 1\mod{3}$ which divide $a$. Set $F:=\Q(E[3])$, $G:=\op{Gal}(F/\Q)$ and $A\subset G$ be the subset of $\sigma\in G$ such that 
\[\op{det}\rho_{E,3}(\sigma)=1\text{ and } \op{trace} \rho_{E,3}(\sigma)\neq 2.\]

\begin{lemma}\label{dumb lemma}
    With respect to notation above, let $\ell\nmid 6a$ be a prime number. Then $\ell$ is unramified in $\Q(E[3])$ and the following are equivalent 
    \begin{enumerate}
        \item $\ell\equiv 1\mod{3}$ and $\widetilde{E}(\F_\ell)[3]=0$, 
        \item $\op{Frob}_\ell\in A$.
    \end{enumerate}
\end{lemma}
\begin{proof}
    Since $\ell\nmid 6a$, it is a prime of good reduction and it follows from the Neron--Ogg--Shafarevich theorem that $\ell$ is unramified in $\Q(E[3])$. We then note that 
    \[\op{det}\rho_{E,3}(\op{Frob}_\ell)=\ell \text{ and }\op{trace}\rho_{E,3}(\op{Frob}_\ell)=\ell+1-\# \widetilde{E}(\F_\ell).\] The result thus follows.
\end{proof}

\begin{lemma}\label{lemma F cap L =K}
    With respect to the notation above, we have that $F\cap \mathbb{L}=K$.
\end{lemma}
\begin{proof}
Set $H:=\op{Gal}(F/K)\subset G$ and $H':=\op{Gal}(F/F\cap \mathbb{L})\subseteq H$. We show that $H'=H$. Note that $G\simeq \rho_{E,3}(G)$ and $H\simeq \rho_{E,3}(H)$, and hence we may view these as subgroups of $B$. Since $\omega=\op{det}\rho_{E,3}$, we find that $G$ contains a matrix with determinant $-1$. Note that $B=\langle -\op{Id}\rangle \times B_0$, where $B_0$ consists of matrices of the form $\mtx{\ast}{\ast}{0}{1}$. Note that $B_0$ contains $U$ as its unique normal subgroup of order $3$, and contains $3$ subgroups $C_1$, $C_2$ and $C_3$ that are of order $2$. The possibilities for $G$ are $B$, $B_0$, $C_i$ and $\langle -\op{Id}\rangle \times C_i$ as $i=1, 2,3$. Note that since $-\op{Id}\in G$, the groups $U$ and $\langle -\op{Id}\rangle \times U$ are excluded from our list. All abelian quotients of these groups are $2$-groups. Hence, $G/H'=\op{Gal}(F\cap \mathbb{L}/\Q)$ is an abelian $2$-group. Since the abelian $2$-quotient of $\op{Gal}(\mathbb{L}/\Q)$ is $\op{Gal}(K/\Q)$, it follows that $F\cap \mathbb{L}=K$.
\end{proof}

\begin{proposition}\label{Q_a density proposition}
    Assume that $a$ is not a square in $K$ and is prime to $3$. Then, with respect to notation above, the set of primes $\mathcal{Q}_a$ has positive density given by
    \[\mathfrak{d}(\mathcal{Q}_a)=\frac{1}{8 \times 3^{s+1}},\] where $s$ is the number of prime divisors $q$ of $a$ such that $q\equiv 1\pmod{3}$. 
\end{proposition}
\begin{proof}
Recall that $F:=\Q(E[3])$, $G:=\op{Gal}(F/\Q)$ and $H:=\op{Gal}(F/K)$.
The representation
\[\rho=\rho_{E, 3}:\op{G}_\Q\rightarrow \op{GL}_2(\F_3),\]
induces a natural isomorphism between $G$ and the image of $\rho$. Let $\ell\nmid 6a$ be a prime, then, $\ell$ is unramified in $F$. Denote by $\op{Frob}_\ell\in G$ a Frobenius element of $\ell$. Recall that $A$ is the subset of $G$ consisting of $\sigma$ such that 
    \[\op{trace}\left(\rho(\sigma)\right)\neq 2\text{ and }\op{det} \left(\rho(\sigma)\right)=1.\]By Lemma \ref{dumb lemma}, the following are equivalent
    \begin{enumerate}
        \item $\ell\equiv 1\mod{3}$ and $\widetilde{E}(\F_\ell)[3]=0$, 
        \item $\op{Frob}_\ell\in A$. 
    \end{enumerate}
    It is clear that $A$ is contained in $H$. Since $a$ is assumed to not be square in $K$, we find that $E(K)[3]=0$. By Lemma \ref{E[3] image lemma}, $-\op{Id}\in A$. Given any matrix $h\in H$, we have that $\op{det}(\rho(h))=1$ and $\op{trace}(\rho(h))$ is either $1$ or $2$. Note that the trace of a matrix with determinant $1$ cannot be $0$. If $h\in H$, then since $-\op{Id}\in H$, we have that $-h\in H$. Thus, we find that \begin{equation}\label{A=H/2}\#A=\frac{\#H}{2}.\end{equation}
    \par We compute the density of $\mathcal{Q}_a$ by applying the Chebotarev density theorem to the field $F\cdot \mathbb{L}$. Note that $F \cap \mathbb{L}=K$ by Lemma \ref{lemma F cap L =K}. Thus, the Galois group $\op{Gal}(F\cdot \mathbb{L}/K)\simeq H\times \op{Gal}(\mathbb{L}/K)$. Let $\mathcal{S}$ be the subset $A\times \{1\}$ contained in $\op{Gal}(F\cdot \mathbb{L}/K)$ and view $\mathcal{S}$ as a subset of $\op{Gal}(F\cdot \mathbb{L}/\Q)$. Note that the set $\mathcal{S}$ is stable with respect to the action of conjugation. Let $\ell$ be a prime which is unramified in $F\cdot \mathbb{L}$ and such that $\op{Frob}_\ell\in \mathcal{S}$. Then, by Lemma \ref{dumb lemma 2}, the prime $\ell$ splits completely in $\mathbb{L}$, and consequently, conditions (1)--(3) of Definition \ref{def of Q_a} are satisfied. Also $\op{Frob}_\ell\in A$, it follows from Lemma \ref{dumb lemma} that (4) is also satisfied. Thus, by the Chebotarev density theorem and \eqref{A=H/2}, we find that 
    \[\mathfrak{d}(\mathcal{Q}_a)=\frac{\# \mathcal{S}}{[F\cdot \mathbb{L}:\Q]}=\frac{\# \mathcal{S}}{2[F\cdot \mathbb{L}:K]}=\frac{\# A }{2 \# H \times [\mathbb{L}:\Q]}=\frac{1}{4[\mathbb{L}:\Q]}=\frac{1}{8 \times 3^{s+1}}.\]
\end{proof}
We prove the main result of this article.
\begin{proof}[Proof of Theorem \ref{main thm of paper}]
   It is easy to see that $a\notin (K^\times)^2$. The result is a direct consequence of Corollary \ref{Q_a corollary} and Proposition \ref{Q_a density proposition}.
\end{proof}

\begin{corollary}
    Let $E$ and $K$ be as in Theorem \ref{main thm of paper}. Then, there is a positive density of primes $\ell$ such that 
    \[\op{rank} E_{\ell^2 a}(K)=0\text{ and } \op{rank} E_{\ell^4 a}(K)=0.\]
\end{corollary}
\begin{proof}
    The result is an immediate consequence of Theorem \ref{main thm of paper}, upon setting $m=\ell$ to be a prime number in $\mathcal{Q}_a$. 
\end{proof}
\subsection{An illustrative example}\label{illustrative example section}
\par Let us illustrate the result in Theorem \ref{main thm of paper} through an explicit example. Consider the curve $E_{-1}:y^2=x^3-1$. We check that $\op{Sel}_3(E_{-1}/K)=0$. The quadratic twist of $E$ by $-3$ is given by $E_{27}$. It suffices to show that 
\[\op{Sel}_3(E_{-1}/\Q)=0\text{ and }\op{Sel}_3(E_{27}/\Q)=0.\]This is computed via the following Magma code
\begin{verbatim}
E:=EllipticCurve([0,0,0,0,-1]);
ThreeSelmerGroup(E);
Et := EllipticCurve([0,0,0,0,27]);
ThreeSelmerGroup(Et);

Ouput:
Abelian Group of order 1.
Abelian Group of order 1
\end{verbatim}
 






\par Note that $\mathcal{Q}_{-1}$ (cf. Definition \ref{def of Q_a}) consists of the primes $\ell\equiv 1\mod{18}$ such that $\widetilde{E}_{-1}(\F_\ell)[3]=0$. Corollary \ref{Q_a corollary} then asserts that if $m$ is divisible only by primes of $\mathcal{Q}_{-1}$, then $E_{-m^2}(K)$ and $E_{-m^4}(K)$ both have rank $0$. Proposition \ref{Q_a density proposition} asserts that the density of $\mathcal{Q}_{-1}$ is $\frac{1}{24}$. We arrive at the following concrete application of Theorem \ref{section 4 basic thm} and Proposition \ref{Q_a density proposition} for the elliptic curve $E:=E_{-1}$. 

\begin{theorem}\label{rank zero main theorem 4.10}
    Let $E_{/\Q}$ be the elliptic curve $y^2=x^3-1$. Given a cube-free integer $m$, let $E_{(m)}$ (resp. $E_{(m^2)}$) denote the cubic twist of $E$ defined by $y^2=x^3-m^2$ (resp. $y^2=x^3-m^4$). Let $\mathcal{Q}:=\mathcal{Q}_{-1}$ be the set of primes $\ell \equiv 1\mod{18}$ such that $\widetilde{E}(\F_\ell)[3]=0$. The set $\mathcal{Q}$ is infinite with density $1/24$. Furthermore, if $m$ is divisible by only the primes in $\mathcal{Q}$, then, \[\op{rank} E_{(m)}(K)=0 \text{ and } \op{rank} E_{(m^2)}(K)=0.\]
\end{theorem}
\begin{proof}
    The result follows from Theorem \ref{section 4 basic thm} and Proposition \ref{Q_a density proposition}, see above for more details.
\end{proof}

We remark that there are 15 primes in $\mathcal{Q}$ that are $<1000$, given as follows \[\mathcal{Q}\cap [1, 1000]=\{19,127, 163, 199, 271, 307, 379, 487, 523, 631, 739, 811, 883, 919, 991\}.\] The Theorem \ref{rank zero main theorem 4.10} implies in particular that for any of the primes $\ell\in \mathcal{Q}\cap [1, 1000]$ above, the ranks of $E_{-\ell^2}(K)$ and $E_{-\ell^4}(K)$ are both zero. This has been checked by the authors using Magma. 
\subsection{Numerical data}\label{numerical data section}
\par In the table \ref{only table} below (on the next page), we collect some numerical data for the values of $a\in [-20, 20]$ satisfying the conditions of Theorem \ref{main theorem of section 3}, along with the density of $\mathcal{Q}_a$. The data shows that Theorem \ref{main thm of paper} applies to \[a\in \{-17, -16,-14, -10, -9, -8, -6, -5, -1, 6,7,8, 13, 14, 20 \}.\] Thus, the result applies to $15$ out of a total of $34$ values considered for $a$. All computations were done on Magma.
\clearpage
\begin{table}[h]
    \centering
    \begin{tabular}{|c|c|c|c|c|c|} \hline
      \multirow{2}{*}{$a$} & \multirow{2}{*}{$\op{Sel}_3(E_a/\Q)$}  & \multirow{2}{*}{$\op{Sel}_3(E_{-27a}/\Q)$} & \multirow{2}{*}{$\op{Sel}_3(E_a/K)$} & \multirow{2}{*}{$s$} & \multirow{2}{*}{$\mathfrak{d}(\mathcal{Q}_a)$}\\
      & & & & & \\ \hline
       $-20$ & $\Z/3\Z$  & $\Z/3\Z$ & $(\Z/3\Z)^2$ & $-$ & $-$ \\ \hline
       $-19$ & $\Z/3\Z$  & $\Z/3\Z$ & $(\Z/3\Z)^2$ & $-$ & $-$ \\ \hline
       $-18$ & $\Z/3\Z$  & $\Z/3\Z$ & $(\Z/3\Z)^2$ & $-$ & $-$ \\ \hline
       $-17$ & $0$  & $0$ & $0$ & $0$ & $1/24$ \\ \hline
       $-16$ & $0$  & $0$ & $0$ & $0$ & $1/24$ \\ \hline
      $-15$ & $\Z/3\Z$  & $\Z/3\Z$ & $(\Z/3\Z)^2$ & $-$ & $-$ \\ \hline
      $-14$ & $0$  & $0$ & $0$ & $1$ & $1/72$ \\ \hline
      $-13$ & $\Z/3\Z$  & $\Z/3\Z$ & $(\Z/3\Z)^2$ & $-$ & $-$ \\ \hline
      $-11$ & $\Z/3\Z$  & $\Z/3\Z$ & $(\Z/3\Z)^2$ & $-$ & $-$ \\ \hline
      $-10$ & $0$  & $0$ & $0$ & $0$ & $1/24$ \\ \hline
      $-9$ & $0$  & $0$ & $0$ & $0$ & $1/24$ \\ \hline
      $-8$ & $0$  & $0$ & $0$ & $0$ & $1/24$ \\ \hline
      $-7$ & $\Z/3\Z$  & $\Z/3\Z$ & $(\Z/3\Z)^2$ & $-$ & $-$ \\ \hline
      $-6$ & $0$  & $0$ & $0$ & $0$ & $1/24$ \\ \hline
      $-5$ & $0$  & $0$ & $0$ & $0$ & $1/24$ \\ \hline
      $-4$ & $\Z/3\Z$  & $\Z/3\Z$ & $(\Z/3\Z)^2$ & $-$ & $-$ \\ \hline
      $-2$ & $\Z/3\Z$  & $\Z/3\Z$ & $(\Z/3\Z)^2$ & $-$ & $-$ \\ \hline
      $-1$ & $0$  & $0$ & $0$ & $0$ & $1/24$ \\ \hline
      $2$ & $\Z/3\Z$  & $\Z/3\Z$ & $(\Z/3\Z)^2$ & $-$ & $-$ \\ \hline
      $3$ & $\Z/3\Z$  & $\Z/3\Z$ & $(\Z/3\Z)^2$ & $-$ & $-$ \\ \hline
      $5$ & $\Z/3\Z$  & $\Z/3\Z$ & $(\Z/3\Z)^2$ & $-$ & $-$ \\ \hline
      $6$ & $0$  & $0$ & $0$ & $0$ & $1/24$ \\ \hline
      $7$ & $0$  & $0$ & $0$ & $1$ & $1/72$ \\ \hline
      $8$ & $0$  & $0$ & $0$ & $0$ & $1/24$ \\ \hline
      $10$ & $\Z/3\Z$  & $\Z/3\Z$ & $(\Z/3\Z)^2$ & $-$ & $-$ \\ \hline
      $11$ & $\Z/3\Z$  & $\Z/3\Z$ & $(\Z/3\Z)^2$ & $-$ & $-$ \\ \hline
      $12$ & $\Z/3\Z$  & $\Z/3\Z$ & $(\Z/3\Z)^2$ & $-$ & $-$ \\ \hline
       $13$ & $0$  & $0$ & $0$ & $1$ & $1/72$ \\ \hline
       $14$ & $0$  & $0$ & $0$ & $1$ & $1/72$ \\ \hline
       $15$ & $(\Z/3\Z)^2$  & $(\Z/3\Z)^2$ & $(\Z/3\Z)^4$ & $-$ & $-$ \\ \hline
       $17$ & $(\Z/3\Z)^2$  & $(\Z/3\Z)^2$ & $(\Z/3\Z)^4$ & $-$ & $-$ \\ \hline
       $18$ & $\Z/3\Z$  & $\Z/3\Z$ & $(\Z/3\Z)^2$ & $-$ & $-$ \\ \hline
        $19$ & $\Z/3\Z$  & $\Z/3\Z$ & $(\Z/3\Z)^2$ & $-$ & $-$ \\ \hline
      $20$ & $0$  & $0$ & $0$ & $0$ & $1/24$ \\ \hline
    \end{tabular}
    \caption{Numerical data for Theorem \ref{main thm of paper}}
    \label{only table}
\end{table}
\bibliographystyle{alpha}
\bibliography{references}
\end{document}